\documentclass[a4, 11pt] {amsart} 
\usepackage{mathrsfs}
\usepackage{amsmath}
\usepackage{amssymb}
\usepackage[all]{xy}
\usepackage[dvips]{graphicx, color}
\usepackage{amsmath}
\usepackage{amssymb}
\usepackage[all]{xy}
\usepackage{geometry}
\usepackage[dvips]{graphicx, color}
\usepackage{mathrsfs}  
\usepackage{enumitem}

\begin{document}

\numberwithin{equation}{section}
\newtheorem{thm}[equation]{Theorem}
\newtheorem{pro}[equation]{Proposition}
\newtheorem{prob}[equation]{Problem}
\newtheorem{qu}[equation]{Question}
\newtheorem{cor}[equation]{Corollary}
\newtheorem{con}[equation]{Conjecture}
\newtheorem{lem}[equation]{Lemma}
\theoremstyle{definition}
\newtheorem{ex}[equation]{Example}
\newtheorem{defn}[equation]{Definition}
\newtheorem{ob}[equation]{Observation}
\newtheorem{rem}[equation]{Remark}

\hyphenation{homeo-morphism} 

\newcommand{\calA}{\mathcal{A}} 
\newcommand{\calD}{\mathcal{D}} 
\newcommand{\calE}{\mathcal{E}}
\newcommand{\calC}{\mathcal{C}} 
\newcommand{\Set}{\mathcal{S}et\,} 
\newcommand{\Top}{\mathcal{T}\!op \,}
\newcommand{\Topst}{\mathcal{T}\!op\, ^*}
\newcommand{\calK}{\mathcal{K}} 
\newcommand{\calO}{\mathcal{O}} 
\newcommand{\calS}{\mathcal{S}} 
\newcommand{\calT}{\mathcal{T}} 
\newcommand{\Z}{{\mathbb Z}}
\newcommand{\C}{{\mathbb C}}
\newcommand{\Q}{{\mathbb Q}}
\newcommand{\R}{{\mathbb R}}
\newcommand{\N}{{\mathbb N}}
\newcommand{\F}{{\mathcal F}}
\newcommand{\PP}{{\mathbb P}}
\def\op{\operatorname}
\hfill

\title[Ranks of %rational 
homotopy and cohomology groups] 
{Ranks of %rational 
homotopy and cohomology groups 
for rationally elliptic spaces and algebraic varieties 
}

\author{Anatoly Libgober and Shoji Yokura}

\thanks{
\\
}

\date{}

\address{Department of Mathematics, University of Illinois, Chicago, IL 60607}

\email{libgober@math.uic.edu}

\address{
Graduate School of Science and Engineering, Kagoshima University, 1-21-35 Korimoto, Kagoshima, 890-0065, Japan
}

\email{yokura@sci.kagoshima-u.ac.jp}

\maketitle

\begin{abstract}
 We discuss inequalities between the values of \emph{homotopical and cohomological
 Poincar\'e polynomials} of the self-products of rationally elliptic spaces. 
For rationally elliptic quasi-projective
varieties, we prove
inequalities between the values of generating functions for the ranks of
the graded pieces of the weight and Hodge filtrations of the canonical mixed Hodge structures on
homotopy and cohomology groups.  
Several examples of such mixed
Hodge polynomials and related inequalities for rationally elliptic quasi-projective algebraic varieties
are presented. 
One of the consequences is that the homotopical
(resp. cohomological) mixed Hodge
polynomial of a rationally elliptic toric manifold is a sum
(resp. a product) of polynomials of projective
spaces. We introduce an invariant called
  \emph{stabilization threshold} $\frak{pp} (X;\varepsilon)$ for a
  simply connected rationally elliptic space $X$ and a positive real number $\varepsilon$, and we show that the
  Hilali conjecture implies that $\frak{pp} (X;1) \le 3$.
\end{abstract}

%%%%%%%%%%%%%%
\section{Introduction}

A rationally elliptic space is a simply connected topological space $X$ such that 
$$\text{$\dim \left (\pi_*(X) \otimes \Q \right) < \infty$ and $\dim H^*(X;\Q) < \infty$}$$
 where $\pi_*(X) \otimes \Q := \sum_{i\ge1} \pi_i(X) \otimes \Q$ and  $H^*(X;\Q) :=\sum_{j \ge 0} H^j(X;\Q)$.
This interesting class of spaces has received considerable
attention, but a complete picture of structure, geometry or
invariants of spaces in this class appears to be far from clear.
Very strong restrictions on the ranks of homotopy group
were found a long time ago by J. B. Friedlander and S. Halperin (
see \cite{FH} and also
\cite{FHT} or \cite{FOT}). To recall them, let $x_i$ (resp. $y_j$) denote a basis of
$\pi_{\op{odd}}(X)\otimes \Q$ (resp. $\pi_{\op{even}}(X) \otimes \Q$) and let 
$n$
be the formal dimension of the space $X$, i.e., the maximal degree 
$n$ such that 
$H^n(X;\Q) \ne 0.$
We set
$$\pi_{\op{even}}(X)\otimes \Q:=\sum_{k \ge 1} \pi_{2k}(X) \otimes \Q, \, \pi_{\op{odd}}(X)\otimes \Q:=\sum_{k \ge 0} \pi_{2k+1}(X)\otimes \Q,$$
$$H^{\op{even}}(X;\Q):=\sum_{k \ge 0} H^{2k}(X;\Q), \, H^{\op{odd}}(X;\Q):=\sum_{k \ge 0} H^{2k+1}(X;\Q).$$
Then we have the following:
\begin{enumerate}[label=(\alph*)]
\item \label{a} $\sum_i\op{deg} x_i \le 2n-1, \sum_j \op{deg} y_j \le n$.
\item \label{b} $n=\sum_i  \op{deg} x_i-\sum_j (\op{deg} y_j-1)$.
\item $\chi^{\pi}(X):=\dim \left (\pi_{\op{even}}(X)\otimes \Q \right) -\dim
  \left (\pi_{\op{odd}}(X)\otimes \Q \right) \le 0$. 
\item $0 \le \chi(X)=\dim H^{\op{even}}(X;\Q)- \dim H^{\op{odd}}(X;\Q)$. 
\item $\chi(X) >0 \Longleftrightarrow \chi^{\pi}(X)=0$.
\item \label{poincare} Betti numbers $b_i =\dim H_i(X;\mathbb Q)$ of $X$ satisfy Poincar\'e duality \cite[\S 38
  Poincar\'e Duality]{FHT}.  In particular $b_n=1$ and $b_{n-1}=b_1=0$.
\item Betti numbers satisfy inequalities: $b_m \le {1 \over 2}{n
    \choose m}, m \ne 0,n$ (cf.  \cite[Corollary to Theorem 1]{Pa}) \footnote { This inequality implies that $\dim H^*(X;\mathbb Q) \le 2^{n-1}+1$, which is sharper than $\dim H^*(X;\mathbb Q) \le 2^n$ (\cite[Theorem 2.75]{FOT}).}
\end{enumerate}
Moreover,  
the Hilali conjecture 
\cite{Hil} (also see \cite{HM, HM2}), which is still open, suggests
that:
\begin{equation}\label{hilali}
\op{dim} \left (\pi_*(X) \otimes \mathbb Q \right ) \le \op{dim} H^*(X;\mathbb Q).
\end{equation}

The present paper, instead of (\ref{hilali}), shows different types of inequalities between the ranks of homotopy and
cohomology groups of rationally elliptic spaces (cf. 
\cite{Yo}). They are stated in terms of the 
\emph{cohomological Poincar\'e polynomial} and the \emph{homotopical Poincar\'e polynomial}. 
 For a simply connected rationally elliptic space $X$ we put
$$P_X(t) := \sum_{k \ge 0} \op{dim} H^k(X;\mathbb Q) t^k \quad \text{and} \quad P^{\pi}_X(t) := \sum_{k \ge 2} \op{dim} (\pi_k(X)\otimes \mathbb Q) t^k.$$
%%%%%%%%%
In \cite{Yo} the second named author showed that there exists a positive integer $n_0$
such that for all $n>n_0$ one has $P^{\pi}_{X^n}(1)<P_{X^n}(1)$. Here $X^n= \underbrace{X \times \cdots \times X}_n$ is the Cartesian product of $n$ copies of $X$. Below 
 we show the following (announced in \cite{Yo2}):
%%%%%%%%%%%%%%%%%
\begin{thm}\label{semi-global} Let $X$ be a
simply connected rationally elliptic space. For any positive real number
  $\varepsilon$ there exists
 a positive integer $n(\varepsilon)$ such that for all 
 $n \ge
 n(\varepsilon)$ and all 
 $t \ge \varepsilon$
\begin{equation}\label{var}
P^{\pi}_{X^n}(t) < P_{X^n}(t). 
\end{equation}
\end{thm} 
%%%%%%%%%%%%
\begin{rem}\label{remark-0} Note that, since $X$ is simply connected, $P_X(t)=1$
implies that $X$ is rationally homotopy equivalent to a point
(cf. \cite[Theorem 8.6]{FHT}), and
hence $P_X^{\pi}=0$. In particular, the inequality (\ref{var}) is satisfied 
with $n(\varepsilon)=1,  \forall \varepsilon>0$.  Therefore, in Theorem
\ref{semi-global}  we assume that $P_X(t)>1$. 
We also note that the formal dimension of a simply connected space is bigger than or equal to $2$.
\end{rem}
%%%%%%%%%%%

Theorem \ref{semi-global} suggests the following invariant of a rationally elliptic
homotopy type:
%%%%%%%%%%%%%%
\begin{defn}\label{threshold} \emph{The stabilization threshold} is the smallest integers
  $n(\varepsilon)$
such that inequality (\ref{var}) takes place for 
all $n \ge n(\varepsilon)$.
\end{defn}
%%%%%%%%%%%
We denote the stabilization threshold
by 
  $\frak{pp}(X; \varepsilon)$, where $\frak{pp}$ stands for ``Poincar\'e polynomial". For 
  example, for $\varepsilon =1$ we have 
\begin{enumerate}
\item 
$\frak{pp}(S^{2n+1};1)=1$,
\item $\frak{pp}(S^{2n};1)=3$,
\item $\frak{pp}(\mathbb CP^1,1)=3$ and $\frak{pp}(\mathbb CP^n,1)=2$ if $n \ge 2$. 
\end{enumerate}
%%%%%%%%%%%%%%%%%%%
In terms of this invariant, the inequality of 
Theorem
\ref{semi-global} 
implies the following:
%%%%%%%%%%%%%%%%%
\begin{cor}\label{mainineq}
 For any $\varepsilon>0$ and $r \ge \frak{pp}(X; \varepsilon)$ 
we have 
\begin{equation*}
 r \left (\sum_{i=2}^n \dim \left (\pi_i(X)\otimes \Q \right ) \varepsilon^i \right) < \left (1+ \sum_{i=2}^l \dim H^i(X,\Q)\varepsilon^i \right)^r
\end{equation*}
where $n$ (resp. $l$) is the degree of homotopical
(resp. cohomological) Poincare polynomial.
\end{cor}
%%%%%%%%%%%%%%%
Note that inequality (\ref{hilali}) is a special case of
Corollary (\ref{mainineq}) for the spaces with stabilization threshold
$\frak{pp}(X;1)=1$, but not for the spaces with $\frak{pp}(X;1) \ge 2$. 
The argument used in the proof of Theorem \ref{semi-global}
 is an elementary calculus observation 
and based only on the difference in behavior of homotopy groups and 
cohomology groups in products.

Several results on stabilization threshold and specific values in some 
examples are presented 
in  
Sections \ref{mainsection}, \ref{bounds} and \ref{examples} respectively, 
but let us point out 
that we have the following result about the upper bound of the stabilization threshold $\frak{pp}(X;1) $:
%%%%%%%%%%%%%%%
%%%%%%%%%%%%%%

\begin{thm}\label{formal-form} Let $X$ be a simply connected rationally elliptic
space of formal dimension $n \ge 3$. Then   
\begin{equation*}\label{up}
\frak{pp}(X;1) \le n.
\end{equation*}
\end{thm}

We also show that the Hilali conjecture implies sharp bound, independent
of dimension:
%%%%%%%%%%%%%
\begin{thm}\label{hil-form} 
If a simply connected rationally elliptic space $X$ satisfies the Hilali conjecture, then we have
$$\frak{pp}(X;1) \le 3.$$
\end{thm}
%%%%%%%%%%%
%%%%%%%%%

In particular, the question if $3$ is an unconditional bound of the threshold $\frak{pp}(X;1)$ is 
a weakening of the Hilali conjecture. Note (see Corollary \ref{20}) that the threshold $\frak{pp}(X;1)$ does not exceed $3$ if the formal dimension does not exceed $20$ since the Hilali conjecture is verified up in this range (see \cite{CatMil}).
The Hilali conjecture is also valid for formal spaces (see \cite{HM}), hence the stabilization threshold $\frak{pp}(X;1)$ does not exceed $3$ also for, e.g., the following spaces, which are formal:
\begin{itemize}
\item compact K\"ahler manifolds \cite{DGMS}, 
\item projective varieties with isolated normal singularities with high connectivity of links \cite{CC}, and
\item smooth quasi-projective manifolds with pure Hodge structure (by Dupont's ``purity implies formality" theorem \cite{Du}). 
\end{itemize}

%%%%%%%%%%%%

Now, let $X$ be a quasi-projective algebraic variety.  
Both the homotopy and the cohomology groups carry mixed
Hodge structures (
\cite{De1}, \cite{De2}, \cite{Mo},
\cite{Hain 1},\cite{Hain 2}, \cite{Nav}), which are functorial for regular maps. 
An invariant of these mixed Hodge structures is given by the generating
functions for the dimensions of graded pieces of Hodge and weight
filtrations as follows:
%%%%%%%%%%%%%%%%%
\begin{equation}\label{homological}
MH_X(t,u,v) :=
\sum_{k,p,q} \dim \Bigl ( Gr_{F^{\bullet}}^{p} Gr^{W_{\bullet}}_{p+q} H^k (X;\mathbb C)  \Bigr) t^{k} u^p v^q, 
\end{equation}
where $(W_{\bullet}, F^{\bullet})$ is the mixed Hodge structure of the cohomology groups.
\begin{equation*}
MH^{\pi}_X(t,u,v) :=
\sum_{k,p,q} \dim  \Bigl (Gr_{\tilde F^{\bullet}}^{p} Gr^{\tilde W_{\bullet}}_{p+q} ((\pi_k(X)\otimes \mathbb C)^{\vee})\Bigr ) t^ku^p v^q,
\end{equation*}
%%%%%%%%%%%%%%5
where $(\widetilde W_{\bullet}, \widetilde F^{\bullet})$ is the mixed Hodge structure of the dual of homotopy groups.
They will be called respectively the \emph{cohomological mixed Hodge polynomial} and the \emph{homotopical mixed Hodge polynomial} of $X$.
A 
refinement of 
Theorem \ref{semi-global} (announced in \cite{Yo2}) for algebraic varieties
is as follows:
%%%%%%%%%%%%5
\begin{thm}\label{mixedcase}
 Let $\varepsilon$ and $r$ be positive real numbers such that $\varepsilon <r$ and let $\mathscr C_{\varepsilon,r}:=[\varepsilon, r] \times [\varepsilon,r] \times [\varepsilon,r] \subset (\mathbb R_{\ge 0})^3$ be the cube of size $r - \varepsilon$.
Let $X$ be a rationally elliptic quasi-projective variety. Then there exists a positive integer $n_{\varepsilon, r}$ such that for all $n \ge n_{\varepsilon, r}$
the following strict inequality holds:
$$MH^{\pi}_{X^n}(t,u,v) <  MH_{X^n}(t,u,v)$$
for $\forall (t,u,v) \in \mathscr C_{\varepsilon,r}.$ 
\end{thm}
%%%%%%%%%%
Similarly to $\frak{pp}(X;1)$, we can consider the smallest integer $n_0$ such that for $\forall n \ge n_0$ the following holds
$$MH^{\pi}_{X^n}(t,u,v) < MH_{X^n}(t,u,v)  \quad \forall t \ge a, \forall u \ge b, \forall v \ge c.$$
We denote it by $\frak{mhp}(X; a,b,c)$, where $\frak{mhp}$ stands for ``mixed Hodge polynomial".

Actual calculations of homotopy and
cohomology groups of rationally elliptic quasi-projective varieties
are rather sparse with the main focus being on low dimensional cases
(e.g., see \cite{amoros}, \cite{BMM} and \cite{Herr} where such rationally elliptic spaces are identified)
and even less is known about their 
mixed Hodge theory refinements.
Therefore, besides inequalities, we include several examples, in
particular toric varieties and 
arrangements of linear subspaces and calculate the stabilization thresholds for them.

It would be interesting to find non-trivial\footnote{A trivial example is $X \times C$ where $X$ is any rationally elliptic smooth or singular variety and $C$ is a rational cuspidal curve (which is homeomorphic to $S^2$).} examples of singular algebraic varieties which are rationally elliptic and study for their mixed homotopy and homology polynomials and 
their stabilization thresholds.

In \S 2 we prove Theorems \ref{semi-global} and 
\ref{mixedcase} and several results on stabilization thresholds. Theorems \ref{formal-form} and \ref{hil-form} are proven in \S 3. In the final \S 4 we give explicit calculations of the homotopical and cohomological mixed Hodge polynomials of several compact and open manifolds, including some toric varieties and complement to arrangements of linear subspaces in affine space. In this section we also introduce and discuss homotopical $E$-function which is an analog of classical cohomological $E$-function.

%%%%%%%%%%%%%%%%%%%%5
\section{Proofs of the main results}\label{mainsection}

The isomorphisms $\pi_i(X\times Y)=\pi_i(X) \oplus \pi_i(Y)$ and the
K\"unneth formula $H^n(X \times Y, \mathbb Q)= \sum_{i+j=n} H^i(X; \mathbb Q) \otimes H^j(Y; \mathbb Q)$ imply that the homotopical Poincar\'e polynomial $P^{\pi}_X(t)$ and the cohomological Poincar\'e polynomial $P_X(t)$ are respectively additive and multiplicative, i.e., 
$$P^{\pi}_{X \times Y} (t) =P^{\pi}_X(t) + P^{\pi}_Y(t) \quad \text{and} \quad  P_{X \times Y} (t) =P_X(t) \times  P_Y(t),$$
which imply that 
Theorem \ref{semi-global} is an immediate 
consequence of the following:
%%%%%%%%%%%%%5
\begin{lem}\label{general thm} Let $\varepsilon$ be a positive real number. Let $P(x)$ and $Q(x)$ be two polynomials of the following types:
$$P(x) = \sum_{k=2}^p a_kx^k, \,  a_k \ge 0, \quad \quad Q(x) = 1 + \sum_{k=2}^q b_kx^k, \quad b_k \ge 0, \, b_q \not = 0 .$$
Then there exists a positive integer $n_0$ such that for $\forall n \ge n_0$
\begin{equation}\label{keyineq}
n P(x) < Q(x)^n \, \, (\forall x \ge \varepsilon).
\end{equation}
\end{lem}
%%%%%%%%%%
\begin{rem}
For our purpose it is sufficient to consider $b_q=1$,  but we do not assume it.
\end{rem}
%%%%%%%%%%%%%%%
\begin{proof} [{\bf Proof} of Lemma \ref{general thm} ]
 Select a positive integer $N_0$ such that $\op{deg} \left( Q(x)^{N_0}
 \right) > \op{deg} \left ( N_0P(x) \right) $ and take $s_0>1, s_0 \in
 \R$ such that 
$Q(x)^{N_0} >N_0P(x)$ for any $x>s_0$.  Then $R(s,r)$ defined by $R(s,r):=Q(s)^r-rP(s)$ we have the following for
all  $r \ge N_0$ and all $s > s_0$:
$$\frac{\partial R(s, r)}{\partial r}= \log Q(s) \cdot Q(s)^r-P(s)> \log Q(s) \cdot N_0P(s)-P(s),$$
which is
positive for all $s>\op{max}(s_0,e)$ since $\log Q(s)>1$, because $Q(s) \ge Q(e) = 1+ \sum_{k=2}^q b_ke^k >e$ since $b_q \not = 0$. Thus for all $s>\op{max}(s_0,e)$ the function 
$R(s,r)$ is increasing with respect to $r$ and $R(s,N_0) = Q(s)^{N_0} - N_0P(s) >0$, thus, in particular  
$R(x,n)=Q(x)^n -nP(x) >0$ for all $x>\op{max}(s_0, e)$ and for all $n \ge N_0$. 
Therefore we have that 
$$ \text{$nP(x) < Q(x)^n$ for
all $x>\op{max}(s_0, e)$ and for all $n \ge N_0$. 
}$$ 
Now, we have the following 
$$\lim_{n \rightarrow \infty} {{nP(\xi)}\over {Q(\xi)^n}}= P(\xi) \lim_{n \rightarrow \infty} {{n}\over {Q(\xi)^n}}=0$$  
for any fixed
$\xi \in [\varepsilon,s_0]$, since $Q(\xi)>1$ for $\xi>0$.  
Therefore, we see that there exists an integer $n(\xi)$ such that for all $n>n(\xi)$
one has $nP(\xi) < Q(\xi)^n$. Having such an integer $n(\xi)$ for each $\xi$, we  can find 
$\delta_{\xi}$ such that for $\vert x-\xi \vert <\delta_{\xi}$ 
and $n>n(\xi)$ one has 
$nP(\xi) <Q(\xi)^n$. Selecting a finite set of 
$\xi_i$ such that the intervals of length $\delta_{\xi_i}$ centered at
$\xi_i$ cover
$[\varepsilon,s_0]$, we see that for 
$N \ge \op{max}\{n(\xi_i),N_0 \}$
one has (\ref{keyineq})  for all $x \ge \varepsilon$. 
\end{proof}
%%%%%%%%%%%%%%
\begin{rem} Let $\frak n(\varepsilon,P,Q)$ be the smallest integer $n_0$
  satisfying conditions of Lemma \ref{general thm}. We can find an upper bound $\frak u$ of the threshold $\frak n(\varepsilon,P,Q)$, i.e., $\frak n(\varepsilon,P,Q) \le \frak u$,
 as follows.

(A) First we consider the case when $0 < \varepsilon \le 1$: 
Let $m$ be the number of the monomials $a_{l_i}x^{l_i} (1 \le i \le m)$ in $P(x)$ 
and $b_qx^q$ be the top degree term of $Q(x)$. Let $\frak u_i$ be an upper bound of the stabilization threshold $\frak n(\varepsilon,ma_{l_i}x^{l_i},1+b_qx^q)$, i.e., $\frak n(\varepsilon,ma_{l_i}x^{l_i},1+b_qx^q) \le \frak u_i$, and let $\frak u:=\op{max}\{\frak u_1, \cdots, \frak u_m\}$.
Then for $\forall n \ge \frak u$
we have for all $i$:
\begin{equation*}\label{twotermineq}
 n(ma_{l_i}x^{l_i})<(1+b_qx^q)^n \quad \forall x \ge \varepsilon
\end{equation*}
and hence
\begin{equation*}nP(x) = {n \over m}\sum_{i=1}^m
  ma_{l_i}x^{l_i} < {1 \over m}\sum_1^m(1+b_qx^q)^n < \bigl(1 + b_kx^k + \cdots b_qx^q \bigr)^n
  =Q(x)^n.
\end{equation*}
Therefore we get that $\frak n(\varepsilon,P,Q) \le \frak u$.

Now, each upper bound $\frak u_i$ of the threshold $\frak n(\varepsilon,ma_{l_i}x^{l_i},1+b_qx^q)$ is obtained as follows, by considering the inequality $n(ma_ix^{l_i})<(1+b_qx^q)^n$ for (a) $x >1$ and (b) $\varepsilon \le x \le 1$ :

\noindent 
(a) $x >1$:
\begin{enumerate}
\item Find an integer $s$ such that $sq > \ell_i$ and
    $s \ge 2$ (condition used in the next step),
\item Find an integer $\widehat n_0$ (depending on $a_{l_i},b_q,l_i,s$)
such that $\displaystyle  \frac{ma_{l_i}}{b_q^s} \le \frac{1}{\widehat n_0}
\binom{\widehat n_0}{s}$ for $\widehat n_0 \ge s$, which implies that $\displaystyle  \frac{ma_{l_i}}{b_q^s} \le \frac{1}{n}
\binom{n}{s}$ for $\forall n \ge \widehat n_0 \ge s$. (If $s=1$, then 
$\frac{1}{n} \binom{n}{s} = 1$ for $\forall n$, in which case there might not exist such an integer $\widehat n_0$, depending on the integers $m, a_{l_i},b_q$. )
\end{enumerate}
%%%%%%%%%%%%%%%%%%%
Then, for $\forall n \ge \widehat n_0$:
\begin{equation*}
n(ma_{l_i}x^{l_i})\le {n \choose s}b_q^sx^{\ell_i} <{n \choose s}b_q^sx^{qs} ={n \choose s} (b_qx^q)^s <(1+b_qx^q)^n \quad \text{for $x > 1$}
\end{equation*}

\noindent
(b)  $\varepsilon \le x \le 1$:

First we observe that $x^{l_i} \le 1$ for $\varepsilon \le x \le 1$, hence it suffices to consider the inequality $n(ma_{l_i})<(1+b_qx^q)^n$, which implies that $n(ma_{l_i}x^{l_i})<(1+b_qx^q)^n$.

\begin{enumerate}
\item[(3)] Find a positive integer $\widetilde n_0$ which is larger than the largest of the roots 
of the following equation:
$$(ma_{l_i})y=(1+b_q\varepsilon^q)^y.$$
\end{enumerate}

In order to show the inequality $n(ma_{l_i})<(1+b_qx^q)^n$ for $\forall n \ge \widetilde n_0$ and for $x \in [\varepsilon, 1]$, for  a fixed $u$ we consider the line 
$z=e \log(u)y$, which as direct calculation readily shows, 
is tangent to the curve $z=u^y$ at the point $y_*(u)={1 \over {\log(u)}}$.
Any other line through the origin of $(z,y)$-plain either does not
intersect $z=u^y$ or intersects it at two points. Taking $u=1+b_qx^q$, 
we conclude that if 
$ma_{l_i}<e\log(1+b_qx^q)$, then $(ma_{l_i})y <(1+b_qx^q)^y$ for $\forall y$,  in particular, 
$n(ma_{l_i}) <(1+b_qx^q)^n$ for $\forall n$.
Otherwise, 
$n(ma_{l_i}) <(1+b_qx^q)^n$ is satisfied for $\forall n \ge
y_0(x)$ where $y_0(x)$ is the largest coordinate of intersection 
of the line  
$z=(ma_{l_i})y$ and exponential curve
$z=(1+b_qx^q)^y$. To get an upper bound on $y_0(x), x \in [\varepsilon,
1]$, note that the largest $y$-coordinate of the intersection of the line
$z=(ma_{l_i})y$ with the exponential curve $z=u^y$, is increasing when $u$ is getting
smaller 
and its minimal value is $(1+b_q\varepsilon^q)$,i.e. for $x=\varepsilon$. Hence the upper bound of $y_0(x)$ is 
the largest of the roots 
of the equation $(ma_{l_i})y=(1+b_q\varepsilon^q)^y.$ Therefore, we have that $n(ma_{l_i}) <(1+b_qx^q)^n$ for $\forall n \ge \widehat n_0$., i.e.,
$$n(ma_{l_i}x^{l_i})<(1+b_qx^q)^n \quad \text{for $\forall x \in [\varepsilon, 1]$.}$$
Finally, we let $\frak u_i = \op{max} \{\widehat n_0, \widetilde n_0\}$, then for $\forall n \ge \frak u_i$ we have
$$n(ma_{l_i}x^{l_i})<(1+b_qx^q)^n \quad \text{for $\forall x \ge \varepsilon$.}$$

(B) In the case when $\varepsilon >1$: We do the same thing as in (a), just by replacing $x >1$ by $x \ge \varepsilon$. Then we let $\frak u_i := \widehat n_0$.

\end{rem}
%%%%%%%%%%%%

%%%%%%%%%%%%%
\begin{rem}\label{form-product} We have the following inequality for the stabilization thresholds: 
\begin{equation}\label{pp-product}
 \frak{pp}(X\times Y;\varepsilon) \le \op{max} \{\frak{pp}(X;\varepsilon),\frak{pp}(Y;\varepsilon) \}
\end{equation}
for a positive real number $\varepsilon$ such that $P_X(\varepsilon) \ge 2$ and $P_Y(\varepsilon) \ge 2$.
Indeed, we let $\frak{pp}(X;\varepsilon):=n_X$ and $\frak{pp}(Y;\varepsilon):=n_Y$ , then we have
$$nP^{\pi}_X(t) < P_X(t)^n \quad \forall n \ge n_X, \forall t \ge \varepsilon,$$
$$nP^{\pi}_Y(t) < P_Y(t)^n, \quad \forall n \ge n_Y, \forall t \ge \varepsilon.$$
Then for $\forall n \ge \op{max}\{n_X, n_Y \}$ and $\forall t \ge \varepsilon$ we have
\begin{equation}\label{equ-1}
n(P^{\pi}_X(t) + P^{\pi}_Y(t) ) < P_X(t)^n + P_Y(t)^n.
\end{equation}
Since $P_X(t)$ and $P_Y(t)$ are increasing functions and $P_X(\varepsilon) \ge 2$ and $P_Y(\varepsilon) \ge 2$, $P_X(t) \ge 2$ and $P_Y(t) \ge 2$ for $\forall t \ge \varepsilon$. Hence we have
\begin{equation}\label{equ-2}
P_X(t)^n + P_Y(t)^n \le P_X(t)^n \cdot P_Y(t)^n = \left ( P_X(t) \cdot P_Y(t)\right )^n.
\end{equation}
$P_X(t)^n + P_Y(t)^n \le P_X(t)^n \cdot P_Y(t)^n$ follows from that 
$$P_X(t)^n \cdot P_Y(t)^n - P_X(t)^n - P_Y(t)^n = \left (P_X(t)^n -1 \right) \left (P_Y(t)^n -1 \right) -1 \ge 0$$
because $P_X(t)^n -1 \ge 1 $ and $P_Y(t)^n -1 \ge 1$ for $\forall t \ge \varepsilon$. Therefore it follows from (\ref{equ-1}) and (\ref{equ-2}) that 
$nP^{\pi}_{X\times Y}(t) < P_{X \times Y}(t)^n$ for $\forall n \ge \op{max}\{n_X, n_Y \}$ and $\forall t \ge \varepsilon$. Therefore we get $\frak{pp}(X\times Y;\varepsilon) \le \op{max} \{\frak{pp}(X;\varepsilon),\frak{pp}(Y;\varepsilon) \}$. However, in general we have $\frak{pp}(X\times Y;\varepsilon) \not = \op{max} \{\frak{pp}(X;\varepsilon),\frak{pp}(Y;\varepsilon) \}$. For example, we can see that $\frak{pp}(S^{2n};1)=3$, but $\frak{pp}(S^{2n} \times S^{2n};1)=2$.
\end{rem}
%%%%%%%%%%%%%
Now we will turn to comparison of 
the homotopical and cohomological mixed Hodge
polynomials.

In fact the cohomological mixed Hodge polynomial is also multiplicative just like the (cohomological) Poincar\'e polynomial $P_X(t)$
%%%%%%%%%%%%
\begin{equation*}\label{mh-multi}
MH_{X \times Y}(t,u,v) =MH_X(t,u,v)  \times MH_Y(t,u,v)
\end{equation*}
%%%%%%%%%%%%
which follows from the fact that the mixed Hodge structure is compatible with the tensor product (e.g., see \cite{PS}.)
On the other hand the homotopical mixed Hodge polynomial is additive just like the homotopical Poincar\'e polynomial $P^{\pi}_X(t)$

\begin{equation*}\label{mh-pi-additive}
MH^{\pi}_{X \times Y}(t,u,v) = MH^{\pi}_X(t,u,v) + MH^{\pi}_Y(t,u,v)
\end{equation*}
since $\pi_{*}(X \times Y) = \pi_{*}(X) \oplus \pi_{*}(Y)$ and the category of mixed Hodge structures is abelian and the direct sum of a mixed Hodge structure is also a mixed Hodge structure. 
%%%%%%%%%%%
In this paper the following special multiplicativity and additivity are sufficient:
\begin{equation}\label{mh-pi-multi}
MH_{X^n}(t,u,v) = \{MH_X(t,u,v)\}^n,
\end{equation}
\begin{equation}\label{mh-pi-additive}
MH^{\pi}_{X^n}(t,u,v) = nMH^{\pi}_X(t,u,v).
\end{equation}
%%%%%%%%%%%%%%

In fact, in a similar way to that of Theorem \ref{general thm}, using multiplicativity and additivity relations (\ref{mh-pi-multi})  and (\ref{mh-pi-additive}), 
we can show the following proposition. Let $\mathbb R_{> 0}$ be the set of positive real numbers.
%%%%%%%%%%%%%
\begin{pro}\label{prop} Let $(s,a,b) \in (\mathbb R_{>0})^3$. Let $X$ be a rationally elliptic quasi-projective variety.
Then there exists a positive integer $n_{(s,a,b)}$ such that for 
$\forall n \ge n_{(s,a,b)}$
the following strict inequality holds
$$MH^{\pi}_{X^n}(t,u,v) <  MH_{X^n}(t,u,v)$$
for $|t-s| \ll 1, |u-a| \ll 1, |v-b| \ll 1$.
\end{pro}
%%%%%%%%%%%%%%%%
\begin{proof} For the sake of completeness and/or the sake of the reader, we give a proof, which is similar to the proof of Lemma \ref{general thm}.
We set
$$MH^{\pi}_{X}(t,u,v) = \sum_{k\ge 2, p \ge 0, q \ge 0} a_{k,p,q} t^ku^pv^q, \quad MH_X(t,u,v) = 1 +\sum_{k\ge 1, p \ge 0, q \ge 0} b_{k,p,q}t^ku^pv^q.$$
If all the coefficients $b_{k,p,q}=0$, then $H^*(X;\mathbb Q)=\mathbb Q =H^{*}(pt;\mathbb Q)$, which implies (as in Remark \ref{remark-0}) that $X$ is rationally homotopy equivalent to the point, hence $\pi_*(X)={0}$. The above strict inequality automatically holds.
So we can assume that $b_{k_0,p_0,q_0} \not = 0$ for some $(k_0,p_0,q_0)$. Then for $(s,a,b) \in (\mathbb R_{>0})^3$ we have
$$ MH_X(s,a,b) = 1 +\sum_{k\ge 1, p \ge 0, q \ge 0} b_{k,p,q}s^ka^pb^q \ge 1 + b_{k_0,p_0,q_0}s^{k_0}a^{p_0}b^{q_0} >1.$$
Therefore whatever the value of $MH^{\pi}_{X}(s,a,b)$ is, by the same argument as in the proof of Theorem \ref{general thm}, 
$$\lim_{n \to \infty} \frac{n(MH^{\pi}_{X}(s,a,b))}{(MH_X(s,a,b))^n} = MH^{\pi}_{X}(s,a,b) \lim_{n \to \infty} \frac{n}{(MH_X(s,a,b))^n} =0.$$
Hence there exists a positive integer $\widehat{N_{(s,a,b)}}$ 
such that for $\forall n \ge \widehat{N_{(s,a,b)}}$
\begin{equation}\label{opencond}
\frac{n(MH^{\pi}_{X}(s,a,b))}{(MH_X(s,a,b))^n} < 1.
\end{equation}
%%%%%%%%%%%%%%%%%%%%
Equivalently, we have for 
$\forall n \ge \widehat{N_{(s,a,b)}}$
\begin{equation*} \label{MH-ineq-1}
nMH^{\pi}_X(s,a,b) <  MH_X(s,a,b)^n.
\end{equation*}
%%%%%%%%
Now the proof is concluded as the proof of Lemma \ref{general thm}
using openness of condition (\ref{opencond}). 
\end{proof}
%%%%%%%%%%
The following theorem follows from the above proposition and the
compactness of the cube $\mathscr C_{\varepsilon,r}$.
%%%%%%%%%%%%%
\begin{thm} 
 Let $\varepsilon$ and $r$ be positive real numbers such that $\varepsilon <r$ and let $\mathscr C_{\varepsilon,r}:=[\varepsilon, r] \times [\varepsilon,r] \times [\varepsilon,r] \subset (\mathbb R_{\ge 0})^3$ 
Let $X$ be a rationally elliptic quasi-projective variety. Then there exists a positive integer 
$n_{\varepsilon, r}$ such that for all $n \ge n_{\varepsilon, r}$
the following strict inequality holds:
$$MH^{\pi}_{X^n}(t,u,v) <  MH_{X^n}(t,u,v)$$
for $\forall (t,u,v) \in \mathscr C_{\varepsilon,r}.$ 
\end{thm}
%%%%%%%%%%
\begin{rem} In a similar manner to the proof of (\ref{pp-product}) in Remark \ref{form-product}, we can see the following inequality as to the threshold $\frak{mhp}$:
\begin{equation*}
 \frak{mhp}(X\times Y;a, b, c) \le \op{max} \{\frak{mph}(X;a,b,c),\frak{mhp}(Y;a,b,c) \}
\end{equation*}
for positive real numbers $a, b, c$ such that $MH_X(a,b,c) \ge 2$ and $MH_Y(a,b,c) \ge 2$.
\end{rem}
%%%%%%%%%%
\begin{rem} 
We defined in Introduction the stabilization threshold $\frak{pp}(X;\varepsilon)$ as the smallest integer $n_0$ such that for all $n \ge n_0$ the following inequality (\ref{var}) holds: $P^{\pi}_{X^n}(t) <P_{X^n}(t) (\forall t \ge \varepsilon)$. In particular, it takes place for the product space $X^{\frak{pp}(X;\varepsilon)}$. On the other hand this inequality is equivalent to $nP^{\pi}_X(t) < \left (P_X(t) \right)^n (\forall t \ge \varepsilon)$, study of which is a key ingredient for our results. This inequality can be considered without assuming that $n$ is an integer, but for $n$ being a positive real number. The same applies to the stabilization threshold  $\frak{mph}(X;a,b,c)$. Thus we can consider the \emph{real stabilization thresholds} $\frak{pp}_{\mathbb R}(X;\varepsilon)$ and $\frak{mph}_{\mathbb R}(X;a,b,c)$, which are more subtle invariants than the integral ones and are more difficult to analyze. For details on properties and calculation of 
these invariants of pairs of polynomials, 
rational elliptic homotopy types and quasi-projective varieties,  we refer to \cite{LY}.
\end{rem}
%%%%%%%%%%%%%%%%
\section{Bounds for Stabilization Thresholds}\label{bounds}
We will start with a conditional result, which yields unconditional 
bound in small dimensions. 
%%%%%%%%%%
\begin{thm}\label{<3} If a simply connected rationally elliptic space $X$ satisfies the Hilali conjecture, then we have
$\frak{pp}(X;1) \le 3.$
\end{thm}
%%%%%%%%%%
\begin{proof} 
Let $X$ be a 
simply connected
  rationally elliptic space 
of formal dimension $n$. Let the homotopical and cohomological Poincar\'e polynomials of $X$ be
$$P^{\pi}_X(t)=a_2t^2+ \cdots + a_it^i + \cdots + a_{\ell}t^{\ell},$$
$$P_X(t)=1 + b_2t^2 + \cdots + b_kt^k + \cdots + t^{n}.$$
(Note that $a_2=b_2$ by the Hurewicz theorem and recall that $b_n=1$ and $b_{n-1}=b_1=0$.)

First we observe that in order to prove that for a positive integer $\frak m \ge 2$
\begin{equation*}
\frak{pp}(X;1) \le \frak m,
\end{equation*}
it suffices to show that
\begin{equation}\label{m}
\frak m P_X^{\pi}(t) < P_X(t)^{\frak m} \quad (\forall t \ge 1).
\end{equation}
Which implies that 
\begin{equation}\label{m+1}
(\frak m+1) P_X^{\pi}(t) < P_X(t)^{\frak m+1} \quad  (\forall t \ge 1)
\end{equation}
and by induction we get  $mP_X^{\pi}(t) < P_X(t)^m \, (\forall t \ge 1)$ for $\forall m \ge \frak m.$
Indeed, the inequality (\ref{m}) implies
\begin{equation}\label{m+1/m}
(\frak m +1)P_X^{\pi}(t) < (\frak m +1) \left (\frac{1}{\frak m} P_X(t)^{\frak m} \right).
\end{equation}
Now
\begin{align*}
P_X(t)^{\frak m +1} -  (\frak m +1) \left (\frac{1}{\frak m} P_X(t)^{\frak m} \right) & =P_X(t)^{\frak m} \left (P_X(t) - \frac{\frak m +1}{\frak m} \right)\\
& =P_X(t)^{\frak m} \left (1 + b_2t^2 + \cdots t^n - 1 - \frac{1}{\frak m} \right)\\
& \ge P_X(t)^{\frak m} \left (t^n - \frac{1}{\frak m} \right)\\
& \ge P_X(t)^{\frak m} \left (1 - \frac{1}{\frak m} \right)\\
& >0 \quad \text{(since $\frak m \ge 2$)}
\end{align*}
Hence we obtain (\ref{m+1}) by the inequality (\ref{m+1/m}).

Now, we show that
$$3P_X^{\pi}(t) < P_X(t)^3 \quad \forall t \ge 1.$$
%%%%%%%%%%%%
First,  
we need to observe that it follows from \cite[Theorem
 32.15]{FHT} that we have the following bound for the degree $\ell$ of $P^{\pi}_X(t)$:
\begin{equation}\label{2n-1}
\ell \le 2n -1.
\end{equation}
%%%%%%%%%%%%%%
\begin{align*}
& (P_X(t))^3 - 3P^{\pi}_X(t) \\
& = (t^{n} + b_{n-2}t^{n-2} + \cdots +b_2t^2+ 1)^3 - 3(a_{\ell}t^{\ell} + \cdots +a_2t^2) \\
& \ge  (t^{n} + b_{n-2}t^{n-2}
 + \cdots  +b_2t^2+ 1)^3- 3t^{\ell}(a_{\ell} + \cdots +a_2) \quad \text{(since $t^j \ge t^2 (j \ge 2)$ for $\forall t \ge 1$)}\\
& \ge  (t^{n} + b_{n-2}t^{n-2} +  \cdots + b_2t^2+ 1)^3 - 3t^{2n-1}(a_{\ell} + \cdots +a_2) \qquad \text{(by (\ref{2n-1}))}
\end{align*}
The Hilali conjecture is $\op{dim} \left (\pi_*(X) \otimes \mathbb Q \right ) \le \op{dim} H_*(X;\mathbb Q)$, i.e. $P^{\pi}_X(1) \le P_X(1) $, or
\begin{equation*}
a_{\ell} + \cdots +a_2 \le 1+ b_{n-2}+ \cdots + b_2 + 1.
\end{equation*}
Before going furthermore, for the presentation below we point out the following about $1+ b_{n-2} + \cdots + b_2 + 1$:
\begin{enumerate}
\item If $n=2$, $P_X(t) =1 +t^2$ (thus, $P^{\pi}_X(t) =t^2 + \cdots$). Hence $1+ b_{n-2} + \cdots + b_2 + 1=1+1$, thus, the part $b_{n-2}+\cdots +b_2 =0$.
\item If $n=3$, then $P_X(t) =1 + t^3$ (thus, $P^{\pi}_X(t) =t^3 + \cdots$), since it follows from the Poincar\'e duality of Betti numbers (see \ref{poincare} in Introduction) that $b_2=b_1=0$. Hence $1+ b_{n-2} + \cdots + b_2 + 1=1+1$, thus, the part $b_{n-2}+\cdots +b_2 =0$.
\item If $n=4$, then $P_X(t) =1 + b_2t^2 + t^4$, since $b_3=b_1=0$. Hence $1+ b_{n-2} + \cdots + b_2 + 1=1+ b_2+ 1$, thus, the part $b_{n-1}+\cdots +b_2 =b_2$.
\end{enumerate}
With the part $b_{n-2}+\cdots +b_2$ in the cases when $n=2,3,4$ being understood as above,
the above sequence of inequalities continues as follows:
\begin{align*}
& \ge  (t^{n} + b_{n-2}t^{n-2}+ \cdots + b_2t^2 + 1)^3 - 3t^{2n-1}(1+ b_{n-2}+\cdots + b_2 + 1) \\
%\end{align*}
%\begin{align*}
& = \left \{ (t^n +1) + (b_{n-2}t^{n-2} + \cdots +b_2t^2) \right\}^3 - 3t^{2n-1}\left \{ 2 + (b_{n-2}+\cdots + b_2) \right\} \\
& \ge (t^n +1)^3 + 3(t^n +1)^2(b_{n-2}t^{n-2} + \cdots +b_2t^2) - 6t^{2n-1}- 3t^{2n-1}(b_{n-2}+\cdots + b_2) \\
& \ge (t^n +1)^3 - 6t^{2n-1}+ 3(t^n +1)^2(b_{n-2}+ \cdots +b_2) - 3t^{2n-1}(b_{n-2}+\cdots + b_2)  \\
& \ge (t^n +1)^3 - 6t^{2n-1} + 3t^{2n}(b_{n-2}+ \cdots +b_2) - 3t^{2n-1}(b_{n-2}+\cdots + b_2)  \\
& \hspace{8cm} \text{(using $(t^n +1)^2 \ge t^{2n}$)}\\
& = (t^n +1)^3 - 6t^{2n-1}+ 3(t^{2n}- t^{2n-1}) (b_{n-2}+ \cdots +b_2) \\
& \ge  (t^n +1)^3 - 6t^{2n-1}\hspace{3cm} \text{(since $t^{2n} - t^{2n-1} =t^{2n-1}(t-1) \ge 0$)}\\
& \ge  (t^n +1)^3 - 6t^{2n} \hspace{3cm} \text{(again, since $t^{2n} \ge t^{2n-1}$ for $t \ge 1$)}\\
& = (t^n)^3 -3(t^n)^2+3t^n +1 \\
& =(t^n-1)^3 +2 \\
& >0.
\end{align*}
Therefore, $3P_X^{\pi}(t) < P_X(t)^3 \quad \forall t \ge 1.$
\end{proof}
%%%%%%%%%%%%%%

Combining 
Theorem \ref{<3} with the result of  \cite{CatMil} we obtain:

\begin{cor}\label{20} For a rationally elliptic space $X$ of homological dimension
  not exceeding 20, the stabilization threshold $\frak{pp}(X;1)$ is at most 3.
\end{cor}

The next proposition gives unconditional bound on the stabilization threshold, 
depending, however, on the homological dimension.

\begin{pro}\label{<n}
Let $X$ be a simply connected rationally elliptic
space of formal dimension $n \ge 3$. Then we have  
\begin{equation*}\label{up}
\frak{pp}(X;1) \le n.
\end{equation*}
\end{pro}
%%%%%%%%%%
%%%%%%%%%%
The argument below uses, in addition to (\ref{2n-1}),
  the following bound (cf. \cite[Theorem
 32.15]{FHT}):
\begin{equation}\label{<n}
P^{\pi}_X(1)=a_2+ a_3+ \cdots  + a_{\ell} \le n.
\end{equation}

\begin{proof} In order to prove the 
proposition, it suffices to show that
$$nP_X^{\pi}(t) < P_X(t)^n \quad \forall t \ge 1.$$
\begin{align*}
& (P_X(t))^n - nP^{\pi}_X(t) \\
& \ge (t^{n} + 1)^n - n(a_{\ell}t^{\ell} + \cdots +a_2t^2) \\
& \ge (t^{n} + 1)^n - nt^{\ell}(a_{\ell}+ \cdots +a_2) \\
& \ge (t^{n} + 1)^n - n^2t^{2n-1} \qquad \text{(by (\ref{2n-1}) and (\ref{<n})}\\
%\end{align*}
%\begin{align*}
& \ge (t^{n} + 1)^n - n^2t^{2n} \qquad \text{(since $t^{2n} \ge t^{2n-1}$ for $t \ge 1$)}\\
& = \sum_{k=0}^n {n \choose k} t^{nk} - n^2t^{2n} \\
& >  \sum_{k=2}^n {n \choose k} t^{nk} - n^2t^{2n} \\
& \ge  t^{2n}\sum_{k=2}^n {n \choose k} - n^2t^{2n} \qquad \text{(since $t^{nk} \ge t^{2n}$ for $k \ge 2$ and $t \ge 1$)}\\
& = t^{2n} \left \{\sum_{k=2}^n {n \choose k} - n^2 \right \}\\
& = t^{2n} \left \{\sum_{k=0}^n {n \choose k} - {n \choose 1} - {n \choose 0} - n^2 \right \}\\
& = t^{2n} \left \{ 2^n - (n^2+n+1) \right \}\\
& > 0
\end{align*}

assuming that $n \ge 5$. For $n=3,4$ the claim follows from
Corollary \ref{20} above.
\end{proof}
%%%%%%%%%%%%
\begin{rem} For $n=2$ the above proposition does not hold since the formal dimension of $\mathbb CP^1$ is $2$, but $\frak{pp}(\mathbb CP^1;1)=3$.
\end{rem}
%%%%%%%%%%%%%
We conclude this section with the question on ``mixed Hodge polynomial" version of Theorem \ref{<3} and Proposition \ref{<n}. More precisely:
\begin{enumerate}
\item Does there exist a fixed integer $\frak a (\ge 3)$ such that $\frak{mhp}(X;1,1,1) \le \frak a$ for any rationally elliptic quasi-projective variety $X$ satisfying the Hilali conjecture?
\item Does there exist an integer $\frak a(n) (\ge n)$ such that $\frak{mhp}(X;1,1,1) \le \frak a(n)$ for any rationally elliptic quasi-projective variety $X$ of formal dimension $n$?
\end{enumerate}
%%%%%%%%%%%%%%%%%%%
\section{Examples and concluding remarks}\label{examples}
Here we present several explicit calculations of thresholds and introduce and discuss some property of homotopical $E$-function which is an analog of classical cohomological $E$-function.

\subsection{Examples}
The purpose of this section is to provide examples of calculations
of exact values of stabilization thresholds. 

%%%%%%%%%%%
\subsubsection{$\C^{n+1}\setminus 0$}\label{complementtopoint} 
Here $n >0$.
This is a smooth quasi-projective variety, for which 
the mixed Hodge structures on cohomology and homotopy 
can be constructed using log-forms (cf. \cite{De1} and \cite{Mo} resp.).
Since this space can be retracted on $S^{2n+1}$ and the Hurewicz
isomorphism preserves the Hodge structure (cf. \cite{Hain 1}) and
calculating the mixed Hodge structure on $H_n(\C^{n+1}\setminus 0)$ 
(for example using  Gysin exact sequence for the homology of the
complement to smooth divisor on the blow up of $\PP^{n+1}$ at a point)
we obtain:
\begin{equation*} MH_{\mathbb C^{n+1} \setminus
    \{0\}}(t,u,v) = 1 + t^{2n+1}(uv)^{n+1},
\end{equation*}
\begin{equation*} MH^{\pi}_{\mathbb C^{n+1} \setminus
  \{0\}}(t,u,v) = t^{2n+1}(uv)^{n+1}.
\end{equation*}
Hence we have
\begin{equation*}MH_{\mathbb C^{n+1} \setminus \{0\}}(t,u,v) = 1 +
  MH^{\pi}_{\mathbb C^{n+1} \setminus \{0\}}(t,u,v).
\end{equation*}

%%%%%%%%%%%%%%%%
\subsubsection{Projective spaces}
\begin{ex}\label{ex-cpn}  We start with  $X=\mathbb CP^n$. We have
$$P_{\mathbb CP^n}(t)= 1+ t^2 + \cdots +t^{2n} \quad \text{and} \quad P_{\mathbb CP^n}^{\pi}(t)= t^2 + t^{2n+1}.$$
One easily verifies that 
\begin{equation*}
\frak{pp}(\mathbb CP^n;1)=
\begin{cases}3 & if \ \  n=1, \\
   2 & if \ \ n\ge 2.
\end{cases}
\end{equation*}
\end{ex}
The mixed Hodge polynomials are as follows:
\begin{equation*}MH_{\mathbb CP^n}(t,u,v) = 1 + t^2uv + t^4(uv)^2 +
  \cdots +t^{2i}(uv)^i + \cdots + t^{2n}(uv)^n.
\end{equation*}
\begin{equation*}MH_{\mathbb CP^n}^{\pi}(t,u,v) = t^2uv +
  t^{2n+1}(uv)^{n+1}. \hspace{5cm} 
\end{equation*}
The cohomological case is trivial and the claim in the homotopical case
follows using the Hurewicz isomorphism for $\pi_2$ and for higher
homotopy groups the locally trivial fibration $\mathbb C^{\times}
\hookrightarrow \mathbb C^{n+1} \setminus \{0\} \to \mathbb CP^n$, the
calculation in \S \ref{complementtopoint} and the corresponding exact sequence
$$ \cdots \to \pi_{2n+1}(\mathbb C^{\times}) \to \pi_{2n+1}(\mathbb C^{n+1} \setminus \{0\}) \to \pi_{2n+1}(\mathbb CP^n) \to \pi_{2n}(\mathbb C^{\times}) \to \cdots.$$
which is an exact sequence of mixed Hodge structures \cite[Theorem 4.3.4]{Hain 1}.

%%%%%%%%%%%
One easily verifies that:
\begin{enumerate}
\item $\frak{mhp}(\mathbb CP^1;1,1,1)=3.$
\item If $n \ge 2$, then $\frak{mhp}(\mathbb CP^n;1,1,1)=2.$ In fact, this can be made to the following a bit sharper statement: for $\forall m \ge 2$
$$MH^{\pi}_{(\mathbb CP^n)^m}(t,u,v) < MH_{(\mathbb CP^n)^m}(t,u,v)  \quad \text{for} \, \, \forall t \ge 1, \forall (u,v) \,\, \text{such that} \,\, uv \ge 1.$$
\end{enumerate}

%%%%%%%%%%%%%%%
\subsubsection{Compact toric manifolds}
In \cite[Theorem 3.3]{BMM} I. Biswas, V. Mu\~noz and A. Murillo show that the
homological Poincar\'e polynomial of a rationally elliptic toric
manifold coincides with that of a product of complex projective
spaces. Below, using a recent result due to M. Wiemeler \cite{W} we show that the same thing holds for the homotopical Poincar\'e polynomial, in fact, for the homotopical mixed Hodge polynomial, and furthermore we also show that the homological mixed Hodge polynomial of a rationally elliptic toric manifold coincides with that of a product of complex projective spaces, which is a stronger version of the above result of Biswas--Mu\~noz--Murillo:
%%%%
\begin{thm}\label{toric MHP} The homotopical and cohomological  mixed Hodge polynomials of a rationally elliptic toric manifold of complex dimension $n$ coincides with %that 
those of a product of complex projective spaces. To be more precise, if $X$ is the quotient of
$$\prod_{i=1}^k (\mathbb C^{n_i+1} \setminus \{0\})$$
by a free action of commutative algebraic groups, i.e., $(\mathbb C^{\times})^k$. Here $n = \sum_{i=1}^k n_i.$ Then we have
\begin{enumerate}
\item  $MH^{\pi}_X(t,u,v)=MH^{\pi}_{\prod_i^k \mathbb CP^{n_i}} (t,u,v) = \sum_{i=1}^k MH^{\pi}_{\mathbb CP^{n_i}}(t,u,v)$, i.e.,
$$MH^{\pi}_X(t,u,v) = \sum_{i=1}^k \Bigl (t^2uv+t^{2n_i+1}(uv)^{n_i+1} \Bigr) = kt^2uv + \sum_{i=1}^k t^{2n_i+1}(uv)^{n_i+1}.$$
\item $MH_X(t,u,v)=MH_{\prod_i^k \mathbb CP^{n_i}} (t,u,v) = \prod_{i=1}^k MH_{\mathbb CP^{n_i}}(t,u,v)$, i.e.,
$$MH_X(t,u,v) = \prod_{i=1}^k \Bigl (1 + t^2uv + \cdots + t^{2j}(uv)^j + \cdots + t^{2n_i}(uv)^{n_i} \Bigr ).$$
\end{enumerate}
\end{thm}
%%%%%%%%%%
\begin{proof} In \cite{W} M. Wiemeler shows that there is an algebraic isomorphism $X \cong X'$ where $X'$ is the quotient described above:
\begin{equation}\label{quotient}
X' = \Bigl (\prod_{i=1}^k (\mathbb C^{n_i+1} \setminus \{0\}) \Bigr )/ (\mathbb C^{\times})^k.
\end{equation}
(1) First we observe that
\begin{equation*}
\pi_j \Bigl ( \prod_{i=1}^k (\mathbb C^{n_i+1} \setminus \{0\})  \Bigr)\otimes \mathbb Q
= \begin{cases}
\underbrace{\mathbb Q \oplus \cdots \oplus \mathbb Q}_a & \,\, j=2n_i+1, \\
 \qquad \, \, 0 & \,\, j \not = 2n_i+1.
\end{cases}
\end{equation*}
Here $a$ is the number of the same integer $n_i$.
\begin{equation*}
\pi_j \Bigl ( (\mathbb C^{\times})^k \Bigr)\otimes \mathbb Q
= \begin{cases}
\underbrace{\mathbb Q \oplus \cdots \oplus \mathbb Q}_k & \,\, j=1, \\
 \qquad \, \, 0 & \,\, j \not = 1. 
\end{cases}
\end{equation*}
Hence, since each $2n_i+1 \ge 3$, it follows from the long exact sequences of homotopy groups that there is an isomorphism of mixed Hodge structures:
\begin{equation*}
\pi_j(X)\otimes \mathbb Q \cong 
\begin{cases}
\pi_j \Bigl ( \prod_{i=1}^k (\mathbb C^{n_i+1} \setminus \{0\})  \Bigr)\otimes \mathbb Q  & \quad j=2n_i+1,\\
\pi_1 \Bigl ( (\mathbb C^{\times})^k \Bigr)\otimes \mathbb Q = \underbrace{\mathbb Q \oplus \cdots \oplus \mathbb Q}_k, & \quad j=2, \\
 \qquad  \qquad  \qquad \, \, 0 & \quad j \not = 2, j=2n_i+1.
\end{cases}
\end{equation*}
Then it follows from the proof in the above  Example \ref{ex-cpn} that we have the isomorphism of mixed Hodge structures
\begin{equation*}
\pi_j(X)\otimes \mathbb Q \cong 
\begin{cases}
\pi_j \Bigl ( \prod_{i=1}^k \mathbb CP^{n_i} \Bigr)\otimes \mathbb Q  & \quad j=2, 2n_i+1,\\
 \qquad  \qquad  \qquad \, \, 0 & \quad j \not = 2, j=2n_i+1.
\end{cases}
\end{equation*}

Therefore we have 
$$MH^{\pi}_X(t,u,v)=MH^{\pi}_{\prod_i^k \mathbb CP^{n_i}} (t,u,v) = \sum_{i=1}^k MH^{\pi}_{\mathbb CP^{n_i}}(t,u,v).$$
%%%%%%%%%%%%%%%

(2) It follows from \cite{W}  that $X'$ is a so-called Bott manifold, i.e., there is a sequence of fiber bundles over complex projective spaces with a complex projective space as a fiber:
$$X'=B_k \xrightarrow {p_k} B_{k-1} \to \cdots \to B_i \xrightarrow {p_i} B_{i-1} \to \cdots B_2 \xrightarrow {p_1} B_1 \to B_0= \{pt\}$$
where $p_1:B_1 =\mathbb C^{n_1+1} \to B_0=\{pt\}$ and each $p_i:\mathbb P(\mathbb C^{n_i+1} \times B_{i-1}) \to B_{i-1}$ is the projection map of the projectivization $\mathbb P(\mathbb C^{n_i+1} \times B_{i-1})$ of the product $\mathbb C^{n_i+1} \times B_{i-1}$ or a Whitney sum of trivial complex line bundles over $B_{i-1}$. This sequence is sometimes called a Bott tower. Note that the fiber space of $p_i$ is nothing but the complex projective space $\mathbb CP^{n_i}$. Then it follows from Deligne's degeneration of Leray spectral sequence (see \cite{PS}) that for each projection map $p_i:B_i \to B_{i-1}$ the cohomology of $B_i$  with mixed Hodge structure is the tensor product of the cohomology of the base $B_{i-1}$ and the fiber $\mathbb CP^{n_i}$ with mixed Hodge structures. Therefore the mixed Hodge polynomial $MH_X(t,u,v)$ coincides with that of the product of these complex projective spaces: $$MH_X(t,u,v)=MH_{\prod_i^k \mathbb CP^{n_i}} (t,u,v) = \prod_{i=1}^k MH_{\mathbb CP^{n_i}}(t,u,v).$$
\end{proof}
%%%%%%%%%%%%%%%%%
It follows from the above Theorem \ref{toric MHP}
that the cohomological and homotopical Poincar\'e polynomials of a
rationally elliptic toric manifold are the same as those of a product of complex projective spaces, thus as explained in the introduction we get the following:
%%%%%%%%%%%%%
\begin{cor} The Hilali conjecture holds for rationally elliptic toric manifolds. 
\end{cor}
\begin{proof} Since a rationally elliptic toric manifold is formal, thus it follows that the Hilali conjecture holds (see \cite{BMM}). Here we give another simple direct proof, using the above calculation. Let $X$ be a rationally elliptic toric manifold described as (\ref{quotient}). Then it follows from the above Theorem \ref{toric MHP} that we have
$$P^{\pi}_X(1) = MH^{\pi}_X(1,1,1) = \sum_{i=1}^k (1+1) =2k, \quad P_X(1) = MH_X(1,1,1) = \prod_{i=1}^k (1+n_i).$$
Since each $n_i \ge 1$, we have 
$$\prod_{i=1}^k (1+n_i) \ge 2^k.$$
\begin{enumerate}
\item If $k=1$, then $2k=2 =2^1 \le 1+n_1$, thus $P^{\pi}_X(1) \le P_X(1)$.
\item If $k=2$, then $2k=4 =2^2 \le \prod_{i=1}^2 (1+n_i)$, thus $P^{\pi}_X(1) \le P_X(1)$.
\item If $k \ge 3$, then $2k < 2^k \le \prod_{i=1}^k (1+n_i)$, thus $P^{\pi}_X(1) < P_X(1)$.
\end{enumerate}
Therefore, in any case we do have $P^{\pi}_X(1) \le P_X(1)$.
\end{proof}
%%%%%%%%%%%%%%%
\begin{cor}\label{cor-toric} Let $X$ be a rationally elliptic toric manifold and let
$$MH_X(t,u,v)=MH_{\prod_i^k \mathbb CP^{n_i}} (t,u,v) ,\qquad MH^{\pi}_X(t,u,v)=MH^{\pi}_{\prod_i^k \mathbb CP^{n_i}} (t,u,v).$$
If each $n_i \ge 2$, then $\frak{mhp}(X;1,1,1)=2$, and if $n_i=1$ for some $i$, then $\frak{mhp}(X;1,1,1)=3.$
\end{cor}
%%%%%%%%%\
\begin{proof}
$$MH_X(t,u,v)=MH_{\prod_i^k \mathbb CP^{n_i}} (t,u,v) = \prod_{i=1}^k MH_{\mathbb CP^{n_i}}(t,u,v),$$
$$MH^{\pi}_X(t,u,v)=MH^{\pi}_{\prod_i^k \mathbb CP^{n_i}} (t,u,v)= \sum_{i=1}^k MH^{\pi}_{\mathbb CP^{n_i}}(t,u,v).$$
(1) If each $n_i \ge 2$, then it follows from Example \ref{ex-cpn} that 
$$2MH^{\pi}_{\mathbb CP^{n_i}}(t,u,v) < \Bigl (MH_{\mathbb CP^{n_i}}(t,u,v) \Bigr )^2,$$
 hence we have
$$2 \Bigl (\sum_{i=1}^k MH^{\pi}_{\mathbb CP^{n_i}}(t,u,v) \Bigr)  = \sum_{i=1}^k 2MH^{\pi}_{\mathbb CP^{n_i}}(t,u,v) \\\
 < \sum_{i=1}^k \Bigl (MH_{\mathbb CP^{n_i}}(t,u,v) \Bigr)^2.$$
Now,  for $\forall t \ge 1, \forall u \ge 1$ and $\forall v \ge 1$ we have
$$\sum_{i=1}^k \Bigl (MH_{\mathbb CP^{n_i}}(t,u,v) \Bigr)^2 < \prod_{i=1}^k \Bigl (MH_{\mathbb CP^{n_i}}(t,u,v) \Bigr)^2.$$
To show this, first we note that each $MH_{\mathbb CP^{n_i}}(t,u,v) \ge 2$ for $\forall t \ge 1, \forall u \ge 1$ and $\forall v \ge 1$. Then it suffices to show that if each $d_i \ge 4 (1 \le i \le k)$, then
$$d_1 + d_2 + \cdots + d_k < d_1 d_2\cdots d_k.$$
Indeed, this follows by induction. Clearly $a_1+a_2 < a_1a_2$ since 
$$a_1a_2 -(a_1+a_2) = (a_1 -1)(a_2-1) -1 \ge 3\times 3-1>0.$$
Suppose that $d_1 + d_2 + \cdots + d_{k-1} < d_1 d_2\cdots d_{k-1}.$ Then
\begin{align*}
d_1 + d_2 + \cdots + d_{k-1} + d_k & = (d_1 + d_2 + \cdots + d_{k-1}) +d_k  \\
& < d_1 d_2 \cdots d_{k-1} + d_k \\
& <  d_1 d_2 \cdots d_{k-1} d_k \, .
\end{align*}
Therefore we have
$$2 \Bigl (\sum_{i=1}^k MH^{\pi}_{\mathbb CP^{n_i}}(t,u,v) \Bigr) < \Bigl (\prod_{i=1}^k MH_{\mathbb CP^{n_i}}(t,u,v) \Bigr)^2.$$
(2) If $n_i=1$ for some $i$, then it follows from Example \ref{ex-cpn} that $\frak{mhp}(\mathbb CP^1;1,1,1)=3$, i.e.,
$3MH^{\pi}_{\mathbb CP^1}(t,u,v) < \bigl (MH_{\mathbb CP^1}(t,u,v) \bigr )^3$. Surely for the other ones we have
$3MH^{\pi}_{\mathbb CP^{n_j}}(t,u,v) < \bigl (MH_{\mathbb CP^{n_j}}(t,u,v) \bigr )^3$. Hence by the same argument as above we have
$$3 \Bigl (\sum_{i=1}^k MH^{\pi}_{\mathbb CP^{n_i}}(t,u,v) \Bigr) < \Bigl (\prod_{i=1}^k MH_{\mathbb CP^{n_i}}(t,u,v) \Bigr)^3.$$
Hence $\frak{mhp}(X;1,1,1)=3.$
\end{proof}
%%%%%%%%%%%%%%%%
\begin{rem}\label{rem-toric} Even if we fix $u=1$ and $v=1$ in the above proof of Corollary \ref{cor-toric}, we have the same proof, therefore we have that
if each $n_i \ge 2$, then $\frak{pp}(X;1)=2$, and if $n_i=1$ for some $i$, then $\frak{pp}(X;1)=3.$
\end{rem}
%%%%%%%%%%%%%
%%%%%%%%%%%%%%%
\subsubsection{Arrangements of linear subspaces} G. Debongnie
(cf. \cite{deb}) described the structure of arrangements of subspaces
in $\C^n$ which complements are rationally elliptic. If follows that
such complements are products of $\prod_i \left (\C^{n_i+1}\setminus 0 \right )$.
Combining this with  calculation in \S \ref{complementtopoint},  we obtain:

\begin{thm}\label{arrangements} The homotopical and cohomological mixed Hodge
  polynomials of a simply connected rationally elliptic complement $X$
  of an arrangement of linear subspaces are as follows:
\begin{enumerate}
\item  
$MH^{\pi}_X(t,u,v) = MH^{\pi} _{\prod^k_i \left (\C^{n_i+1}\setminus 0 \right )} (t,u,v) =
\sum_{i=1}^k MH^{\pi}_{\C^{n_i+1}\setminus 0}(t,u,v)$, i.e.,  
$$ MH^{\pi}_X(t,u,v) = \sum_{i=1}^k t^{2n_i+1}(uv)^{n_i+1}.$$
\item $MH_X(t,u,v)=MH_{\prod_i^k \left (\C^{n_i+1}\setminus 0 \right )} (t,u,v) = \prod_{i=1}^k MH_{\C^{n_i+1}\setminus 0 }(t,u,v)$, i.e.,
$$MH_X(t,u,v) = \prod_{i=1}^k \Bigl (1 + t^{2n_i+1}(uv)^{n_i+1} \Bigr ).$$
\end{enumerate}
\end{thm}
%%%%%%%%%%%
%%%%%%%%%%%%%%%%
In particular, we obtain: 
\begin{cor}\label{cor-arrangements} In notations of  Theorem
  \ref{arrangements}, we have 
\begin{equation*}
\frak{pp}(X;1)=1 \quad \text{and} \quad \frak{mhp}(X; 1,1,1)=1.
\end{equation*}
\end{cor}
%%%%%%%%%%%%%%%%%%

%%%%%%%%%%%%%%%
%\section{Homotopical and homological E-functions.}
%\subsection{Concluding remarks} \textcolor{magenta}{It is about homotopical E-functions.}
\subsection{Homotopical $E$-function}
Specialization $t=-1$ of the homotopical Poincare polynomial
$MH^{\pi}_X(t,u,v)$ is a homotopical analog of $E$-functions 
(cf. \cite{batyrev}) and is well behaved in several constructions 
described below.

\begin{defn} The homotopical $E$-polynomial $E^{\pi}(X,u,v)$ of a
  complex algebraic variety $X$ which is rational elliptic is defined as follows:
$$E^{\pi}(X,u,v) := MH^{\pi}_X(-1,u,v). $$
\end{defn}
Recall that the homological $E$-function is defined as $E(X,u,v):=
MH^c_X(-1,u,v)$, where one uses in (\ref{homological}) the compactly
supported cohomology. 
%%%%%%%%
 $E(X,u,v)$ satisfies the additivity relation: for an algebraic
 subvariety $Y \subset X$ one has (cf. \cite{batyrev})
\begin{equation}\label{chi-y}
E(X,u,v) =E(Y,u,v) +E(X \setminus Y,u,v).
\end{equation}
This follows from the long exact sequence of compactly supported cohomology groups:
$$ \cdots \to H^k_c(X \setminus Y) \to H^k_c(X) \to H^k_c(Y) \to H^{k+1}_c(X \setminus Y) \to \cdots .$$
Additivity relation for the homotopical $E$-polynomials comes in the
context of locally trivial fibrations 
\begin{equation}\label{loctrivial}F \hookrightarrow E \to B
\end{equation} of
pointed complex algebraic varieties of rationally elliptic $E,F,B$, which induces a long exact sequence of homotopy groups with mixed Hodge structures (see \cite[Theorem 4.3.4]{Hain 1}):
\begin{equation}\label{longexact}
 \cdots \to \pi_k(F) \to \pi_k(E) \to \pi_k(B) \to
  \pi_{k-1} (F) \to \cdots .
\end{equation}
The sequence (\ref{longexact}) yields the following:
%%%%%%%%%%%%%%%%%
\begin{pro} Let $E,F,B$ be simply connected pointed complex algebraic varieties forming 
a locally trivial fibration (\ref{loctrivial}) such that any two of them 
are rationally elliptic. Then we have
$$E^{\pi}(E,u,v) = E^{\pi}(F,u,v)+E^{\pi}(B,u,v).$$
\end{pro}
%%%%%%%%%

%%%%%%%%%%%%%
In the case of homological $E$-polynomials one has multiplicativity in
the case of locally trivial fibrations (\ref{loctrivial})\footnote{for which one does not need to assume that spaces are rationally elliptic}

\begin{thm}(see \cite{CMS, CLMS, CLMS2, MS}) Let $F \hookrightarrow E \to B$ be a smooth complex algebraic fiber bundles. If the fundamental group $\pi_1(B)$ of the base space $B$ acts trivially on the cohomology $H^*(F;\mathbb Q)$ of the fiber space $F$, then, 
$$E(E,u,v) =E(F,u,v) \cdot E(B,u,v).$$
\end{thm}
This is a reformulation of the relation between the Euler
characteristics of bigraded components: $$e^{p,q}(X)=\sum_{k} (-1)^k\dim \Bigl ( Gr_{F^{\bullet}}^{p} Gr^{W_{\bullet}}_{p+q} H^k (X;\mathbb C)  \Bigr) $$
discussed on p.935 of \cite{CLMS2} (the multiplicativity relation is
stated in this paper for $\chi_y$-genus which is a specialization of $E(X,u,v)$.

%%%%%%%%%%%%

Finally we note that a homotopy theoretical analog of additivity relation
(\ref{chi-y}) holds. To state it, recall (cf. \cite{Hain 1}, \cite{Hain 2})
that if $(X,Y)$ is a pair of pointed complex algebraic varieties, the
homotopy groups support a mixed Hodge structure such that the homotopy
exact sequence of the pair $(X,Y)$ is an exact sequence of the mixed
Hodge structures. This sequence implies that for rationally elliptic spaces $X$ and
$Y$ such that $\pi_i(X,Y)=0$ for large $i$ 
\begin{equation*}
E^{\pi}(X,Y,u,v):= MH^{\pi}_{(X,Y)}(-1,u,v) =
\sum_{k} (-1)^k\dim  \Bigl (Gr_{\tilde F^{\bullet}}^{p} Gr^{\tilde W_{\bullet}}_{p+q} ((\pi_k(X,Y)\otimes \mathbb C)^{\vee})\Bigr ) u^p v^q
\end{equation*}
is well-defined and 
the following additivity relation holds:
\begin{equation}\label{additivity}
E^{\pi}(X, u,v) = E^{\pi}(Y,u,v) + E^{\pi}(X,Y,u,v).
\end{equation}

Let $X$ be a compact complex algebraic variety and $Y$ be a closed
subvariety of $X$ such that $X\setminus Y$ is smooth,  then for 
homological $E$-functions one has:
$$E(X \setminus Y,u,v) =E(X, Y,u,v),$$
which shows that additivity (\ref{additivity}) corresponds to (\ref{chi-y}).\\

{\bf Acknowledgements:} We would like to thank the referee and Prof. J. Rosenberg for useful suggestions and comments. S.Y. is supported by JSPS KAKENHI Grant Number JP19K03468. \\

%%%%%%%%%%%%

\end{document}